\documentclass[11pt,english]{extarticle}
\usepackage[T1]{fontenc}
\usepackage[latin9]{inputenc}
\usepackage{geometry}
\usepackage{color}
\usepackage{babel}
\usepackage{float}
\usepackage{units}
\usepackage{mathrsfs}
\usepackage{amsmath}
\usepackage{amsthm}
\usepackage{amssymb}
\usepackage{microtype}
\usepackage[unicode=true,pdfusetitle,
 bookmarks=true,bookmarksnumbered=true,bookmarksopen=false,
 breaklinks=false,pdfborder={0 0 1},backref=false,colorlinks=true]
 {hyperref}
\hypersetup{
 linkcolor=blue, urlcolor=blue, citecolor=blue, pdfstartview={FitH}, hyperfootnotes=false, unicode=true}

\makeatletter
\numberwithin{equation}{section}
\numberwithin{figure}{section}
\theoremstyle{plain}
\newtheorem{thm}{\protect\theoremname}
\theoremstyle{plain}
\newtheorem{lem}[thm]{\protect\lemmaname}
\theoremstyle{definition}
\newtheorem{defn}[thm]{\protect\definitionname}
\theoremstyle{plain}
\newtheorem{prop}[thm]{\protect\propositionname}
\theoremstyle{remark}
\newtheorem{rem}[thm]{\protect\remarkname}
\theoremstyle{plain}
\newtheorem{cor}[thm]{\protect\corollaryname}

\usepackage{lettrine} 
\let\myFoot\footnote
\renewcommand{\footnote}[1]{\myFoot{#1\vspace{3mm}}}
\usepackage{amsmath}
\usepackage{amssymb}
\usepackage{xcolor}
\usepackage{mlmodern}
\usepackage[T1]{fontenc}
\usepackage{amsmath}

\usepackage{tikz}
\usepackage{circuitikz}

\makeatother

\providecommand{\corollaryname}{Corollary}
\providecommand{\definitionname}{Definition}
\providecommand{\lemmaname}{Lemma}
\providecommand{\propositionname}{Proposition}
\providecommand{\remarkname}{Remark}
\providecommand{\theoremname}{Theorem}

\begin{document}
\title{Skorokhod energy of planar domains}
\author{Mrabet Becher \thanks{Institut Pr\'eparatoire aux Etudes des Ing\'enieurs de Monastir, Tunisia.} ,
Maher Boudabra \thanks{King Fahd University of Petroleum and Minerals, KSA.}
, Fethi Haggui \thanks{Institut Pr\'eparatoire aux Etudes des Ing\'enieurs de Monastir, Tunisia.}}
\maketitle
\begin{abstract}
In this work, we introduce the Skorokhod energy of a simply connected
domain. We show that among all domains solving the planar Skorokhod
embedding problem, the Gross solution generates the domain with the minimal
Skorokhod energy.
\end{abstract}

\section{Introduction and main results}

In a pretty recent paper which appeared in $2019$ \cite{gross2019}, the
author R. Gross considered an interesting planar version of the Skorokhod
problem, which was originally formulated in dimension one in $1961$(we refer the reader to \cite{Obloj2004} for a concise survey
of the linear version). The statement of the planar Skorokhod embedding
problem is as follows : Given a non-degenerate  distribution $\mu$ with zero mean
and finite second moment, is there a simply connected domain $U$
(containing the origin) such that if $Z_{t}=X_{t}+Y_{t}i$ is a standard
planar Brownian motion, then $X_{\tau_U}=\Re(Z_{\tau_U})$ has the distribution
$\mu$, where $\tau_U$ is the exit time from $U$. Gross provided an
affirmative answer with a smart and explicit construction of the domain.
In addition, the generated domain has an exit time of finite average.
Gross has also observed that the solution domain is not unique without
further conditions. The core tool of Gross's idea is the following theorem. 
\begin{thm}[L\'evy's theorem]
\label{conformal }  Let $f:U\rightarrow\mathbb{C}$ be a non-constant
analytic function and $Z_{t}$ be a planar Brownian motion running
inside $U$. Then there is a planar Brownian motion $W_{t}$ such
that $f(Z_{t})=W_{\sigma(t)}$ with 
\[
\sigma(t)=\int_{0}^{t}\vert f'(Z_{s})\vert^{2}ds
\]
and $t\in[0,\tau_{U}]$. 
\end{thm}

Theorem \ref{conformal } is commonly known as the conformal invariance
principle. That is, the image of a planar Brownian motion under the
action of a non-constant analytic function is yet another planar Brownian
motion, except it runs at a different speed. With the same notation,
it is equivalent to say that $f(Z_{\sigma^{-1}(t)})$ is planar Brownian
motion. The random time $\sigma(\tau_{U}^{-})$ (left limit of $\sigma(t) $ at $\tau_U$) is referred to as
the projection of the exit time $\tau_{U}$. It is not always the
case that $\sigma(\tau_{U}^{-})$  is the exit time of $f(U)$ unless
$f$ is proper which is the case when $f$ is univalent. A weaker alternative of the properness is the so-called ``$B$-properness'' introduced in \cite{markowsky2015exit}. In particular,
the law of $f(Z_{\tau_{U}})$ is obtainable from that of $Z_{\tau_{U}}$ \cite{markowsky2018distribution}.
The idea of Gross is based on this fact. In order to solve the problem,
it is enough to construct a univalent map acting on the unit disc
such that $\Re(f(e^{\xi i}))$ samples as $\mu$ when $\xi$ is uniformly distributed in $(-\pi,\pi)$. As $Z_{\tau_{\mathbb{D}}}$
is uniformly distributed on the unit circle, then the distribution
of $\Re(f(e^{\xi i}))$ is the same as the distribution of the real part
of a stopped planar Brownian motion on the boundary of $f(\mathbb{D})$. 

In \cite{boudabra2019remarks}, the two authors, Boudabra and Markowsky
gave geometric conditions which, when imposed, make the domain unique. Moreover, the frame of treated distributions has been drastically
enlarged to cover all of finite $p^{th}$ moments where $p$ belongs
to $(1,+\infty)$. Later on, the two authors published a second paper
which generated a new category of domains that solve the Skorokhod
embedding problem. As in their first paper, the new solution applies
to any distribution of finite $p^{th}$ moment ($p>1$). Furthermore,
a uniqueness criterion has been given. We shall use their terminology ``$\mu$-domain''
to tag any simply connected domain that solves the problem. The differences
between the two solutions are summarized as follows :
\begin{itemize}
\item Gross construction :
\begin{itemize}
\item $U$ is symmetric over the real line. 
\item $U$ is $\Delta$-convex, i.e the segment joining any point of $U$
and its symmetric one over the real axis remains inside $U$.
\item If $\mu(\{x\})>0$ then $\partial U$ contains a vertical line segment
(possibly infinite). 
\end{itemize}
\item Boudabra-Markowsky's construction :
\begin{itemize}
\item $U$ is $\Delta^{\infty}$-convex, i.e the upward half ray at any
point of $U$ remains inside $U$.
\item If $\mu(\{x\})>0$ then $\partial U$ contains a downward half ray
at $x$. 
\end{itemize}
\end{itemize}
If $(a,b)$ is a gap for the support of $\mu$, i.e., $(a,b)$ is a null set within the support of $\mu$ \footnote{~It means that the c.d.f of $\mu$ is constant on $(a,b)$.},
then any possible solution would contain the vertical strip $(a,b)\times(-\infty,+\infty)$.
As per the uniqueness, the authors essentially showed that $\Delta^{\infty}$-convexity
or symmetry and $\Delta$-convexity guarantee uniqueness. We provide
the following two examples illustrating the differences of the two
methods. 

\begin{figure}[H]
~~~~~~~~~~~~~~~~~~~~\includegraphics[width=6cm,height=6cm,keepaspectratio]{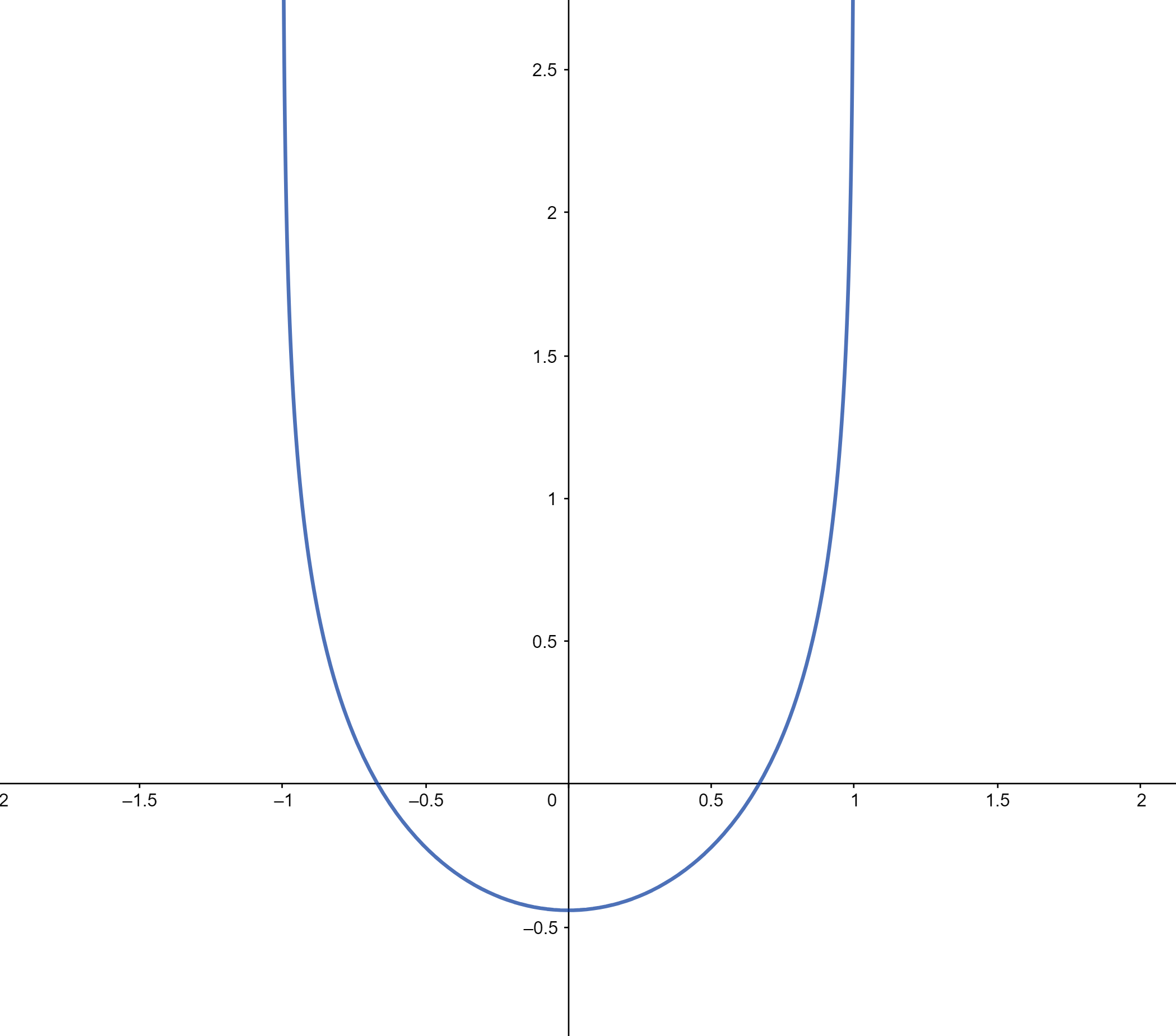}\includegraphics[width=5cm,height=5cm,keepaspectratio]{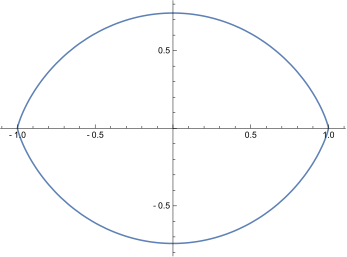}\caption{{\small{}For the uniform distribution on $(-1,1)$, the left domain
is Boudabra-Markowsky's solution while Gross's solution is on the right. }}

\end{figure}

\begin{figure}[H]
~~~~~~~~~~~~~~~~~~~~~~~\includegraphics[width=6cm,height=6cm,keepaspectratio]{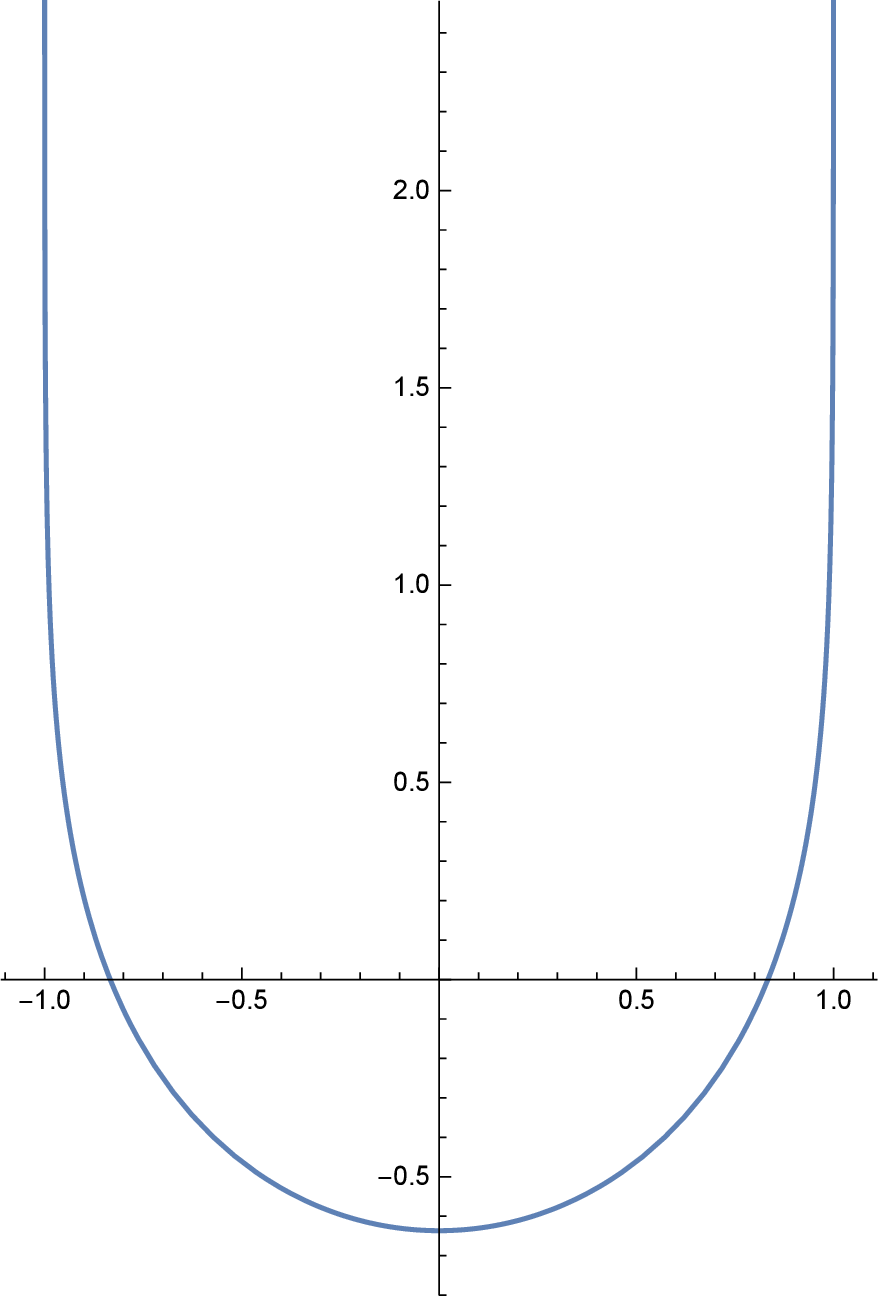}~~~~~~\includegraphics[width=5cm,height=5cm,keepaspectratio]{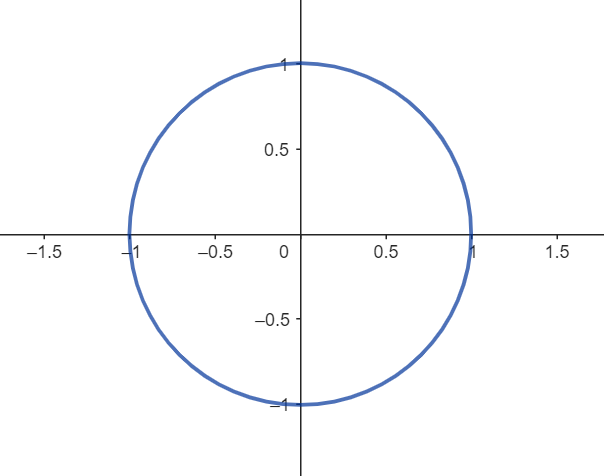}

\caption{{\small{}The domains generated from the shifted-scaled arcsine distribution
$d\mu(x)=\frac{1}{\pi\sqrt{1-x^{2}}}dx$. Gross's method generates the
unit disc while Boudabra-Markowsky's method generates the left domain.}}

\end{figure}

Surprisingly, the second solution turns out to solve completely a
separate optimization problem posed and partially answered in \cite{mariano2020}.
The question therein was to seek the $\mu$-domain having the lowest
or the largest principal Dirichlet eigenvalue. The authors in \cite{mariano2020}
gave the $\mu$-domain of minimal principal Dirichlet eigenvalue for
$\mu$ being the uniform distribution on $(-1,1)$. Their argument
leans on the geometry of the domain they have constructed. In \cite{Boudabra2020},
Boudabra and Markowsky noticed that this argument applies perfectly
to their solution. In particular, they recovered the $\mu$-domain
provided in \cite{mariano2020}. The question about the largest principal
Dirichlet eigenvalue is still open. Throughout this work, the following
notations and terminologies will be adopted : 
\begin{itemize}
\item $\Vert f\Vert_{p}=\left(\frac{1}{2\pi}\int_{0}^{2\pi}\vert f(t)\vert^{p}dt\right)^{\frac{1}{p}}$
where $p>0$.
\item $L_{2\pi}^{p}$ : the space of $2\pi$-periodic functions $f$ such
that $\Vert f\Vert_{p}$ is finite.
\item $W_{2\pi}^{1,p}$ the Sobolev space for the $2\pi$-periodic functions. 
\item The word ``domain'' means a simply connected open region in the plane. Unless otherwise mentioned, the domains considered in this manuscript contain the origin. 
\item $(Z_{t})_{t}$ refers to a standard planar Brownian motion.
\item $\tau_{U}$ is the exit time of $(Z_{t})_{t}$ from a domain $U$. 
\end{itemize} \par
In the literature and in the context of physics, the kinetic energy
of a function $f$, refers to the energy associated with the motion
of particles described by that function. It is a concept commonly
encountered in classical mechanics, where the kinetic energy of an
object is determined by its mass and velocity. In mathematical terms,
the kinetic energy of a (real valued) $2\pi$-periodic function $f$,
which we denote by $\mathfrak{E}(f)$, is proportional to the average
of the squared norm of its velocity. That is
\[
\mathfrak{E}(f):={ \frac{1}{2}}{ \frac{1}{2\pi}}\int_{0}^{2\pi}\vert f'(t)\vert^{2}dt={ \frac{1}{2}}\Vert f'\Vert_{2}^{2}
\]
where $f'$ is the distributional derivative of $f$. In particular,
the energy $\mathfrak{E}(f)$ is infinite for any function $f$ with
jumps. To see this, remark that if $f$ has a discontinuity at $t_{0}$
then $f'(t_{0})=\delta_{t_{0}}\notin L_{2\pi}^{2}$. From Fourier
series theory, $f$ has a Fourier expansion 
\[
f(\theta)=\frac{a_0}{2}+\sum_{n=1}^{+\infty}a_{n}\cos(n\theta)+b_{n}\sin(n\theta)
\]
 and hence 
\[
f'(\theta)=\sum_{n=1}^{+\infty}-na_{n}\sin(n\theta)+nb_{n}\cos(n\theta)
\]
in the distribution sense. Thus, the kinetic energy of $f$ simplifies to
\[
\mathfrak{E}(f)={ \frac{1}{4}}\sum_{n=1}^{+\infty}n^{2}(a_{n}^{2}+b_{n}^{2}).
\]
 When the kinetic energy of $f$ is finite then the Fourier coefficients
decay faster than $\frac{1}{n}$ at infinity. In this context, we
introduce the Skorokhod energy of a domain $U$ containing the origin.
Before that, we state the following Lemma. 

\begin{lem}
\label{well defined} Let $f$ be a univalent function mapping $\mathbb{D}$
onto a domain $U$ and fixing the origin. Then the quantity $\mathfrak{E}(\Re(f(e^{ti})))$
is independent of $f$. 
\end{lem}

\begin{proof}
The existence of a map $f$ is guaranteed by the famous Riemann mapping
theorem. Moreover, any such two functions $f$ and $g$ are related
by some rotation, i.e $f(z)=g(e^{\theta i}z)$. The result follows
by a standard change of variable.
\end{proof}
\newpage
Note that if $$f(z)=\sum_{n=1}^{+\infty}c_{n}z^{n}$$ then 
\[
\mathfrak{E}(\Re(f(e^{ti})))=\frac{1}{4}\sum_{n=1}^{+\infty}n^{2}\vert c_{n}\vert^{2}
\]
where $f(e^{ti})$ is interpreted as the radial limit function
of $f$. In the second section, we will see that such a limit does
always exist in the frame of Hardy spaces and particularly for univalent functions. Now we give the definition of
the Skorokhod energy of a domain $U$. 
\begin{defn}
The Skorokhod energy of $U$, denoted by $\Lambda(U)$, is the quantity
$\mathfrak{E}(\Re(f(e^{ti})))$ where $f$ is any univalent function
mapping $\mathbb{D}$ onto $U$ and fixing the origin.
\end{defn}

By virtue of Lemma \ref{well defined}, the Skorokhod energy of
$U$ is well defined since it does not depend on the choice of the
mapping function. By a slight abuse of notation, we define the Skorokhod energy of an analytic function $$f(z)=\sum_{n=0}^{+\infty}c_n z^n $$ defined on $\mathbb{D}$ by simply setting $$\Lambda(f)=\mathfrak{E}(\Re(f(e^{ti}))).$$ In particular, when $f$ is univalent then $\Lambda(f)=\Lambda(f(\mathbb{D}))$. 

For example, the Skorokhod energy of a centered
disc with radius $r$ is $\frac{r^{2}}{4}$. The next result provides
a criterion to test the finiteness of the Skorokhod energy. 
\begin{prop}
\label{prop:A-finite-energy} A finite Skorokhod energy implies finite perimeter
and area.
\end{prop}

A consequence of Proposition \ref{prop:A-finite-energy} is that unbounded
domains or bounded ones with irregular boundaries are of infinite
Skorokhod energies. In particular, the $\mu$-domains generated by the
technique in \cite{Boudabra2020} have infinite Skorokhod energy as
they are unbounded. Besides, all $\mu$-domains generated from distributions with unbounded support are of infinite Skorokhod energy. Note the property
of having a finite Skorokhod energy is not a size-hereditary property,
i.e if $\Lambda(U)$ is finite and $V\subset U$ then $\Lambda(V)$
is not necessarily finite. 

Now we are almost ahead of our main result. Let $U$ be a domain and
run a standard planar Brownian motion $(Z_{t})_{t}$ inside $U$. Let $\mu_{U}$ be the distribution of $\Re(Z_{\tau_{U}})$,
the real part of $Z_{t}$ killed upon hitting the boundary of $U$. Let $G$ be $\mu_{U}$-domain obtained by applying Gross's method to $\mu_{U}$. In particular, both domains $U$ and $G$ are
$\mu_{U}$-domains. 
\begin{thm}
\label{main theorem} Let $U$ be a domain and let $G$ be the Gross
domain generated by $\mu_{U}$. Then the Skorokhod energy of $G$ is at most that of $U$. In other words, 
\[
\Lambda(G)\leq\Lambda(U).
\]
\end{thm}

Theorem \ref{main theorem} enables us to define the Skorokhod energy
of a distribution $\mu$ to be $\Lambda(G)$ where $G$ is the Gross
domain generated by $\mu$. We shall denote such a quantity by $\Lambda(\mu)$. Now, we provide some explained examples:
\begin{enumerate}
\item By the uniqueness theorem in \cite{boudabra2019remarks}, the unit
disc is the Gross domain of the scaled-shifted arc-sine law given
by $d\mu(x)=\frac{dx}{\pi\sqrt{1-x^{2}}}\mathbf{1}_{\{x\in(-1,1)\}}$.
That is, the Skorokhod energy of $\mu$ is same as $\Lambda(\mathbb{D})$. In particular 
\begin{equation}\label{arcsine}
    \Lambda(\mu)=\frac{1}{4}.
\end{equation}

\item Let $\mu$ be the uniform distribution on $(-1,1)$. From \cite{gross2019},
the corresponding Gross $\mu$-domain is the image of the unit disc
by the map 
\[
\alpha(z)=-\frac{8}{\pi^{2}}\sum_{n=1}^{+\infty}\frac{z^{2n-1}}{(2n-1)^{2}}.
\]
Therefore the Skorokhod energy of $\mu=\text{Uni}(-1,1)$ is 
\begin{equation}\label{unif}
    \Lambda(\mu)=\frac{1}{4}\frac{8^2}{\pi^{4}}\sum_{n=1}^{+\infty}\frac{1}{(2n-1)^{2}}=\frac{1}{4}\frac{8^2}{\pi^{4}}\frac{\pi^{2}}{8}=\frac{2}{\pi^2}.
\end{equation}

\item Let $\mathscr{P}_{m}$ be the regular polygon of vertices $e^{\nicefrac{2\pi ki}{m}}$,
$k=0,...,m-1$ where $m$ is an even integer greater than $2$. A
conformal map from the unit disc onto $\mathscr{P}_{m}$ and fixing
the $e^{\nicefrac{2\pi ki}{m}}$'s, is given by 
\[
f_{m}(z)=\frac{m}{\beta(\nicefrac{1}{m},\nicefrac{(m-2)}{m})}\sum_{n=0}^{\infty}\frac{(\nicefrac{2}{m})_{n}z^{mn+1}}{n!(mn+1)}
\]
where $(x)_{n}:=x(x+1)\cdots(x+n-1)$ and $\beta$ refers to the beta
function (See \cite{markowsky2018distribution}). As $\mathscr{P}_{m}$
fulfills the uniqueness criterion in \cite{boudabra2019remarks}, then it is the Gross domain generated
from the distribution of $\Re(Z_{\tau_{\mathscr{P}_{m}}})$. The Skorokhod
energy of such distribution is infinite as 
\[
\sum_{n=0}^{\infty}n^{2}\left(\frac{(\nicefrac{2}{m})_{n}}{n!(mn+1)}\right)^{2}=+\infty.
\]
\end{enumerate}
\begin{figure}[H]
\begin{centering}
\includegraphics[width=6cm,height=6cm,keepaspectratio]{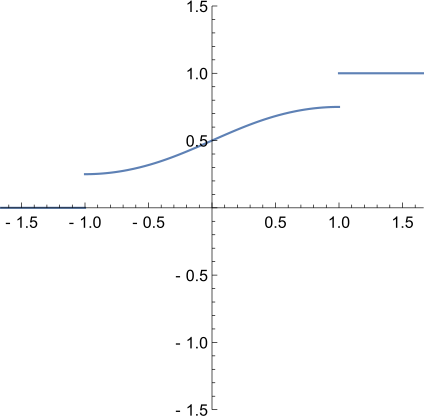}
\par\end{centering}
\caption{{\small{}The graph of the c.d.f of the distribution of $\Re(Z_{\tau_{S}})$
where $S=\frac{2e^{-\frac{\pi i}{4}}}{\sqrt{2}}f_{4}(\mathbb{D})$
is the square of vertices $\{\pm1\pm i\}$. The presence of the two
jumps at $\pm1$ is reflected by the atoms of the distribution.}}
\end{figure}

The original definition of the Skorokhod energy of $\mu$ uses the coefficients of the power series $f$ mapping the unit disc onto the underlying $\mu$-domain $U$. This might represent some computational challenges regarding the coefficients of $f$. However, when the probability distribution $\mu$ is continuous of p.d.f $\rho$, we obtain a closed integral form for the Skorokhod energy of $\mu$. To state the result, we first recall some statistical quantities related to $\mu$. The quantile function of
$\mu$ is defined for every $u\in(0,1)$ by
\begin{equation}
Q(u)=\inf\{x\mid F(x)\geq u\}=\sup\{x\mid F(x)<u\}\,\,\,(a.e).\label{quantile-1}
\end{equation}
where $F$ is the c.d.f of $\mu$, i.e $F(x)=\mu((-\infty,x])$, $x\in \mathbb{R}$. In other words, $Q$ is the pseudo-inverse of $F$. In particular, $Q$ simplifies
to the standard inverse function when $F$ is strictly increasing. The interchange
between $\inf$ and $\sup$ in \ref{quantile-1} is justified by
the fact that the set of discontinuities of $F$ is at most countable.
The quantile function is very handy to generate random variables.
More precisely, if $\xi$ is uniformly distributed on $(0,1)$ then $Q(\xi)$
samples as $\mu$.  
\begin{thm}\label{integral form}
    Assume that $\mu$ has a density $\rho$, and let $Q$ be the quantile function of $\mu$. Then we have
\begin{equation}\label{closed}
   \Lambda(\mu)=\frac{1}{2\pi^2}\int_{0}^{1}\frac{1}{\rho(Q(x))^{2}}dx 
\end{equation} 
provided that $Q'$ has a Fourier series obtained by termwise differentiation of the Fourier series of $Q$.

\end{thm}
To illustrate that \ref{closed} is consistent with the original definition of the Skorokhod energy, we reiterate the earlier examples. 
\begin{itemize}
 \item When $\mu$ is the uniform distribution on $(-1,1)$ then we find $Q(x)=2x-1$. We get 
 \[\frac{1}{2\pi^{2}}\int_{0}^{1}\frac{1}{\rho(Q(x))^{2}}dx=\frac{1}{2\pi^{2}}\int_{0}^{1}4dx=\frac{2}{\pi^{2}}\] which matches the value in \ref{unif}.

\item When $\mu(x)$ is the scaled-shifted arc-sine law then we find $Q(x)=-\cos(\pi x)$. We get 
\[\frac{1}{2\pi^{2}}\int_{0}^{1}\frac{1}{\rho(Q(x))^{2}}dx=\frac{1}{2\pi^{2}}\int_{0}^{1}\pi^{2}(1-\cos(\pi x)^{2})dx=\frac{1}{2\pi^{2}}\frac{\pi^{2}}{2}=\frac{1}{4}\]
which matches the value in \ref{arcsine}.
\end{itemize}
\par
Since Theorem \ref{integral form} provides a closed form for the Skorokhod energy of continuous distributions (with bounded support), it can be used to generate identities involving infinite sums. For example, we can recover the identity $$\sum_{n=1}^{+\infty}\frac{1}{(2n-1)^{2}}=\frac{\pi^2}{8}.$$

\section{Tools and proofs}

In this section, we provide all required ingredients of the results
stated in the previous section. 
\begin{defn}
The Hilbert transform of a $2\pi$- periodic function $f$ is defined
by 
\[
H\{f\}(x):=PV\left\{ \frac{1}{2\pi}\int_{-\pi}^{\pi}f(x-t)\cot(\frac{t}{2})dt\right\} =\lim_{\eta\rightarrow0}\frac{1}{2\pi}\int_{\eta\leq|t|\leq\pi}f(x-t)\cot(\frac{t}{2})dt
\]
where $PV$ denotes the Cauchy principal value, which is required
here as the trigonometric function $t\longmapsto\cot(\cdot)$ has
poles at $k\pi$ with $k\in\mathbb{Z}$. The Hilbert transform is
a bounded operator on $L_{2\pi}^{p}$ with $p>1$. More precisely
we have 
\end{defn}

\begin{thm}
\cite{butzer1971hilbert} If $f$ is in $L_{2\pi}^{p}$ then $H\{f\}$
does exit almost everywhere for $p>1$. Furthermore
\begin{equation}
\Vert H\{f\}\Vert_{p}\leq\lambda_{p}\Vert f\Vert_{p}\label{strong inequality}
\end{equation}
for some positive constant $\lambda_{p}$.
\end{thm}

The case $p=1$ is irregular and unfortunately $H$ becomes unbounded.
For all proofs and properties of the Hilbert transform, we refer the
reader to \cite{butzer1971hilbert,king2009hilbert}. When $p=2$,
which is related to our frame, the Hilbert transform becomes an isometry provided the input function has a zero average.
That is, if the constant term in the Fourier expansion of $f$ is zero then 

\begin{equation}
\Vert H\{f\}\Vert_{2}=\Vert f\Vert_{2}.\label{eq:isometry}
\end{equation}
The equality \ref{eq:isometry} can be easily understood if one
knows that $H\{\cos\}=\sin$ and $H\{\sin\}=-\cos$.
\begin{defn}
Let $f$ be an analytic function acting on the unit disc and $p>0$.
The $p^{th}$ Hardy norm of $f$ is defined by 
\begin{equation}
H_{p}(f):=\sup_{0\leq r<1}\left\{ \frac{1}{2\pi}\int_{0}^{2\pi}|f(re^{\theta i})|^{p}d\theta\right\} ^{\frac{1}{p}}=\sup_{0\leq r<1}N_{r,p}\{f\}.\label{hardy}
\end{equation}
\end{defn}

The Hardy norm is well defined since the quantity $N_{r,p}(f)$ is non
decreasing in terms of $r$ \cite{Rudin2001}. A crucial result about
Hardy norms is that, if $H_{p}(f)$ is finite then $f$ has a radial
extension to the boundary $$f^{\dagger}(\xi):=\lim_{r\rightarrow1}f(r\xi)$$
exists a.e for all $\xi\in\partial\mathbb{D}$. Moreover 
\[
\Vert f^{\dagger}\Vert_{p}=H_{p}(f).
\]
We refer the reader to \cite{duren2000theory} for more details about Hardy spaces. For the sake of brevity, we keep writing $f(e^{ti})$ instead of $f^{\dagger}(e^{ti})$.
In \cite{burkholder1977exit}, the author gave a very fancy result
revealing the tight connections between the theory of planar Brownian
motion and the theory of analytic functions. That is, he showed the
following equivalences 

\begin{equation}\label{equivalence}
   H_{p}(f)<\infty\Longleftrightarrow\mathbf{E}(\tau^{\nicefrac{p}{2}})<\infty\Longleftrightarrow\mathbf{E}(\sup_{0\leq t\leq\tau}\vert Z_{t}\vert^{p})\Rightarrow\mathbf{E}(\vert Z_{t}\vert^{p}) 
\end{equation}

where $f$ is proper and $\tau$ is the exit time of a planar Brownian
motion from $f(\mathbb{D})$. The following theorem concerns the range
of $p$ for which the $p^{th}$ Hardy norm of a univalent function
is finite. In particular, it gives sense to the statement in Lemma \ref{well defined}. 
\begin{thm}
\label{hardy norm} \cite{duren2001univalent} If $f$ is univalent then $H_{p}(f)$ is finite
for all $p<\frac{1}{2}$. 
\end{thm}

The last implication in \ref{equivalence} turns out to be an
equivalence provided that $\mathbf{E}(\ln(1+\tau))<\infty$. However, this
requirement has been shown to be redundant when $f$ is univalent
\cite{boudabra2021some}. Theorem \ref{hardy norm} has triggered the question about the supremum
of $p$ such that $H_{p}(f)$ is finite. Such a number is called the
Hardy number \cite{hansen1970hardy}. As the Hilbert transform maps
$\cos$ to $\sin$ and $\sin$ to $-\cos$, then when $f$ fixes the origin, it is not difficult to see that
\[
f(e^{ti})=\Re\Big(f\big(e^{ti}\big)\Big)+H\left\{ \Re\Big(f\big(e^{ti}\big)\Big)\right\} i.
\]
In other words, the imaginary part of $f(e^{ti})$ is simply the Hilbert
transform of its real part. By \ref{eq:isometry}, we get 
\begin{equation}\label{2.5}
  H_{2}(f)^{2}=\Vert f(e^{ti})\Vert_{2}^{2}=2\left\Vert \Re\Big(f\big(e^{ti}\big)\Big)\right\Vert _{2}^{2}
\end{equation}

In particular
\[
\Lambda(f)=\frac{1}{4}H_{2}(f')^{2}.
\]

\subsection{Proof of Proposition \ref{prop:A-finite-energy}}

The planar version of the isoperimetric inequality states that 
\[
\text{area}(U)\leq{ \frac{1}{4\pi}}\ell(U)^{2}
\]
where $\text{area}(U)$ and $\ell(U)$ are the area of $U$ and its
perimeter \cite{osserman1978isoperimetric,philippin1996isoperimetric}.
That is, it is enough to show that a finite Skorokhod energy implies a finite perimeter. Let $$f(z)=\sum_{n=1}^{+\infty}c_{n}z^{n}$$ be a
univalent function mapping $\mathbb{D}$ onto $U$. The perimeter
of $U$ is 

\[
\ell(U)=2\pi H_{1}(f'),
\]
which covers non-smooth boundaries as well. The Cauchy\textendash Bunyakovsky\textendash Schwarz
inequality yields 
\[
\begin{alignedat}{1}\ell(U)^{2} & \leq2\pi\int_{-\pi}^{\pi}\vert f'(e^{ti})\vert^{2}dt\\
 & = 4\pi^{2}H_{2}(f')^{2}\\
 & \overset{\eqref{2.5}}{=}8\pi^{2}\left\Vert \Re\Big(f'\big(e^{ti}\big)\Big)\right\Vert _{2}^{2}\\
 & =16\pi^{2}\Lambda(U).
\end{alignedat}
\]
The result follows. 

Note that an inequality between the area and the Skorokhod energy
can be obtained without involving the perimeter. In fact, a standard
identity for the area of $U$ is given by 
\begin{equation}
\text{area}(U)=\pi\sum_{n=1}^{+\infty}n\vert c_{n}\vert^{2}.\label{area formula}
\end{equation}
In particular we get 
\[
\text{area}(U)\leq 4\pi \frac{1}{4} \sum_{n=1}^{+\infty}n^{2}\vert c_{n}\vert^{2}=4\pi\Lambda(U).
\]
The three quantities $\text{area}(U),\ell(U),\Lambda(U)$ are related
by the inequality 
\[
\text{area}(U)\leq{ \frac{1}{4\pi}}\ell(U)^{2}\leq4\pi\Lambda(U).
\]
The constants ${\textstyle \frac{1}{4\pi}}$ and $4\pi$ are optimal
for $U=\mathbb{D}$ as all the terms are equal to $\pi$.

\subsection{Proof of Theorem \ref{main theorem}}

The proof of Theorem \ref{main theorem} is based on the symmetric
decreasing rearrangement. For a survey about the subject as well as
the main properties, we refer the reader to \cite{kesavan2006symmetrization}.
Let $f$ be a non negative function and let $\xi_{t}^{f}=\frac{\vert f^{-1}(t,\infty)\vert}{2}$.
When $f$ is clear from the context, we shall omit the superscript.
It is obvious that the map $t\mapsto\xi_{t}$ is non increasing. 
\begin{defn}
The symmetric decreasing rearrangement of $f$, denoted by $f^{*}$,
is defined by
\begin{equation}
f^{*}(x)=\int_{0}^{+\infty}1_{(-\xi_{t},\xi_{t})}(x)dt\label{symm dec rea}
\end{equation}
\end{defn}

The set $(-\xi_{t},\xi_{t})$ is the symmetrization of $f^{-1}(t,\infty)$.
Formula \ref{symm dec rea} shows that the function $f^{*}$ is
even, i.e $f^{*}(x)=f^{*}(\vert x\vert)$. As $s\mapsto1_{(-\xi_{s},\xi_{s})}(\vert x\vert)$ is non-increasing then 
\begin{equation}
f^{*}(\vert x\vert)=\sup\{s\mid\vert x\vert<\xi_{s}={\textstyle \frac{\vert f^{-1}(s,+\infty)\vert}{2}}\}.\label{eq:sup}
\end{equation}

The original definition of $f^{*}$, \ref{symm dec rea}, applies
to non-negative functions; however, the identity \ref{eq:sup} extends
to functions with negative values. That is, if $f$ is a function
bounded below, then we define the symmetric decreasing rearrangement
of $f$ by \ref{eq:sup}. In particular, we find 

\[
f^{*}=g^{*}+\inf(0,\inf f)
\]
with $g=f-\inf(0,\inf f)$. Hence, when $f$ hits the negative real
line, then its symmetric decreasing rearrangement is obtained by a simple
shift. A crucial property of the symmetric decreasing rearrangement
is the equimeasurability. In other words, the level sets $\{x\mid f(x)>t\}$
and $\{x\mid f^{*}(x)>t\}$ have the same Lebesgue measure. From a
probabilistic perspective, this property interprets as : A random
variable $\vartheta$ and its symmetric decreasing rearrangement $\vartheta^{*}$
share the same probability law. Note that $\vartheta$ and $\vartheta^{*}$
are not coupled, i.e., they do not necessarily act on the same sample
space, unless it is symmetric, of course. The next lemma characterizes
a function in terms of its variation and its probability distribution,
i.e the measure of its level sets. 
\begin{prop}
\label{equality }Let $I$ be an interval and let $\vartheta,\varrho:I\rightarrow\mathbb{R}$
be two non-decreasing functions with the same distribution. Then $\vartheta=\varrho$
a.e. 
\end{prop}

\begin{proof}
Since $\vartheta$ and $\varrho$ have the same distribution, then
by monotonicity we get
\[
\begin{alignedat}{1}\vert\vartheta^{-1}(s,\infty)\vert & =\vert\varrho^{-1}(s,\infty)\vert\\
 & =1-\inf\{x\mid\vartheta(x)>s\}\\
 & =1-\inf\{x\mid\varrho(x)>s\}.
\end{alignedat}
\]
Thus $\inf\{x\mid\vartheta(x)>s\}=\inf\{x\mid\varrho(x)>s\}$. By Lemma 11 in \cite{Boudabra2020}, we get $\vartheta=\varrho$ a.e. 
\end{proof}
\begin{rem}
Lemma \ref{equality } can be generalized in the following way : If
$I=\biguplus I_{j}$ (disjoint union of intervals) where $\vartheta,\varrho$ share the same distribution
and have the same monotonicity on each $I_{j}$, then $\vartheta=\varrho$
a.e. 
\end{rem}

In the sequel, $\mu$ is a non-degenerate probability distribution with a bounded support and denote by $F$ its c.d.f. We assume that $\mu$ is centered, i.e has a zero average. The following proposition yields a clue to the relation
between $Q$, the pseudo-inverse of $F$, and its symmetric decreasing rearrangement $Q^{*}$. 
\begin{prop}
\label{Q*=00003DQ(1-2x)} We have 
\begin{equation}
Q^{*}(x)=Q(1-2\vert x\vert)\,\,\,\,(a.e)\label{quantile}
\end{equation}
where $x\in(-\frac{1}{2},\frac{1}{2})$.
\end{prop}

\begin{proof}
We have 
\[
\begin{alignedat}{1}Q^{*}(\vert x\vert) & =\sup\{s\mid\vert x\vert<{\textstyle \frac{\vert Q^{-1}(s,\infty)\vert}{2}}\}\\
 & =\sup\{s\mid\vert x\vert<{\textstyle \frac{1-F(s)}{2}}\}\\
 & =\sup\{s\mid F(s)<1-2\vert x\vert\}\\
 & =Q(1-2\vert x\vert).
\end{alignedat}
\]
\end{proof}
The identity \ref{quantile} is consistent with the definition as
$Q^{*}$ is defined on $(-\frac{1}{2},\frac{1}{2})$ which is the
symmetrization of $(0,1)$. The following corollary follows immediately
from Proposition \ref{Q*=00003DQ(1-2x)}.
\begin{cor}
\label{last corollary} If $\psi:\theta\in(-\pi,\pi)\mapsto\mathbb{R}$
is a random variable sampling as $\mu$ then for almost every $\theta\in(-\pi,\pi)$
\[
\psi^{*}(\theta)=Q(1-{\textstyle \frac{\vert\theta\vert}{\pi}}).
\]
\end{cor}

In \cite{gross2019}, the author showed the following result, which is a cornerstone
of his solution for the planar Skorokhod embedding problem. 
\begin{thm}
Let $(a_{n})_{n}$ be the sequence of Fourier coefficients of $\theta\in(-\pi,\pi)\mapsto Q(\frac{\vert\theta\vert}{\pi})$
then 
\[
\Psi(z)=\sum_{n=1}^{\infty}a_{n}z^{n}
\]
 is univalent on the unit disc. 
\end{thm}

The Fourier expansion of $Q(\frac{\vert\theta\vert}{\pi})$ is simply
$$\sum_{n=1}^{\infty}a_{n}\cos(n\theta)$$ (no $\sin$ terms), where the constant coefficient is zero since the distribution $\mu$ is assumed centered. Gross'
smart idea was to generate the power series out of the Fourier expansion.
Now, notice that the Fourier coefficients of $Q(\frac{\vert\theta\vert}{\pi})$
and $Q(\frac{\vert\cdot\vert}{\pi})^{*}(\theta)=Q(1-\frac{\vert\theta\vert}{\pi})$ differ by a
$(-1)^{n}$ factor. That is, if we denote by $\mathscr{S}(z)$ the
power series generated by the Fourier coefficients of $Q^{*}$ then
we get 
\[
\mathscr{S}(z)=\Psi(-z).
\]
Hence $\mathscr{S}$ is univalent and it obviously maps the unit
disc to $\Psi(\mathbb{D})$, the same $\mu$-domains obtained by Gross'
technique. The last ingredient of the proof of Theorem \ref{main theorem} is the famous  P\'olya-Szeg\"o  inequality. 
\begin{thm}
\cite{kesavan2006symmetrization} If $f$ is in $W^{1,p}$ then $f^{*}\in W^{1,p}$
and we have 
\[
\Vert\nabla f^{*}\Vert_{p}\leq\Vert\nabla f\Vert_{p}.
\]
\end{thm}

 The P\'olya-Szeg\"o  inequality states that symmetric decreasing rearrangement
does not boost the norm of the gradient. In particular, the kinetic
energy of a function does not increase under symmetrization. The norm
of the gradient (up to some factor) appears frequently in areas
related to potential theory like electrostatics \cite{payne1967isoperimetric}. 

Our main result follows now. Let $\Phi$ be any univalent function
mapping $\mathbb{D}$ onto $U$ and fixing the origin. We have 
\[
\Lambda(U)=\mathfrak{E}\bigg(\Re\Big(\Phi\big(e^{\theta i}\big)\Big)\bigg)
\]
and 
\[
\Lambda(G)=\mathfrak{E}\bigg(\Re\Big(\Psi\big(e^{\theta i}\big)\Big)\bigg).
\]
The conformal invariance principle implies that $\Re(\Phi(e^{\theta i})),\Re(Z_{\tau_{U}}),\Re(\Psi(e^{\theta i})),\Re(Z_{\tau_{G}})$
sample as $\mu_{U}.$ By Corollary \ref{last corollary}, the symmetric
decreasing rearrangement of $\Re(\Phi(e^{\theta i}))$ is $Q(1-\frac{\vert\theta\vert}{\pi})$
where $Q$ is the quantile function of $\mu_{U}$. In particular,
the univalent function $\mathscr{S}$ maps the unit disc to the Gross
domain, which is the solution of the Skorokhod problem. By the  P\'olya-Szeg\"o  inequality we get 
\[
\Lambda(G)=\frac{1}{2}\left\Vert \Re\Big(\Psi\big(e^{ti}\big)\Big)'\right\Vert _{2}^{2}=\frac{1}{2}\left\Vert Q({\scriptstyle \frac{\vert\cdot\vert}{\pi}})^{*'}\right\Vert _{2}^{2}\leq\frac{1}{2}\Vert Q({\scriptstyle \frac{\vert\cdot\vert}{\pi}})'\Vert_{2}^{2}=\frac{1}{2}\left\Vert \Re\Big(\Phi\big(e^{\theta i}\big)\Big)\right\Vert _{2}^{2}=\Lambda(U)
\]
which proves our claim. 

\subsection{Proof of Theorem \ref{integral form}}
We shall keep the same notations as in the proof of Theorem \ref{main theorem}. We assume that $\mu$ is of bounded support, since otherwise its Skorokhod energy would be infinite by Proposition \ref{prop:A-finite-energy}. The real part of $\Psi(e^{\theta i})$ coincides with the function $Q(\frac{\vert \theta \vert}{\pi})$ where $\theta \in (-\pi,\pi)$. As the quantile function $Q$ is non-decreasing then it is differentiable a.e by virtue of the monotone differentiation theorem \cite{folland1999real}. As $\pm \pi$ are mapped to the same value by $Q(\frac{\vert \theta \vert}{\pi})$, the requirement on the Fourier series of $Q'$ is satisfied, for instance, when $Q$ is continuous and piecewise smooth \cite{vretblad2003fourier}. Thus we have 

$$
\begin{alignedat}{1}\Lambda(\mu) & =\frac{1}{4\pi}\int_{-\pi}^{\pi}(\Re(\Psi(e^{ti}))')^{2}dt\\
 & =\frac{1}{4\pi}\int_{-\pi}^{\pi}(Q(\textstyle \frac{\vert t\vert}{\pi})')^{2}dt\\
 & =\frac{1}{2\pi}\int_{0}^{\pi}(Q(\textstyle \frac{t}{\pi})')^{2}dt\\
 & =\frac{1}{2\pi^{3}}\int_{0}^{\pi}Q'(\textstyle \frac{t}{\pi})^{2}dt\\
 & =\frac{1}{2\pi^{2}}\int_{0}^{1}Q'(x)^{2}dx\\
 & =\frac{1}{2\pi^{2}}\int_{0}^{1}\frac{1}{\rho(Q(x))^{2}}dx.
\end{alignedat}
$$
Note that the technical condition on the Fourier series of $Q'$ has been checked for the examples listed after Theorem \ref{integral form}. 
\section{Comments and further directions}

Theorem \ref{main theorem} can be seen as follows (with the same
notations): If $$\Phi(z)=\sum_{n=1}^{+\infty}b_{n}z^{n}$$ and $$\Psi(z)=\sum_{n=1}^{+\infty}a_{n}z^{n}$$
then 
\[
\sum_{n=1}^{+\infty}n^{2}\vert a_{n}\vert^{2}\leq\sum_{n=1}^{+\infty}n^{2}\vert b_{n}\vert^{2}.
\]

In a very recent manuscript \cite{boudabra2024brownian}, the
authors investigated the geometric map transforming $U$ to $G$.
They proposed to call it ``Brownian symmetrization'' (inspired by Steiner
symmetrization) where they denoted it by $\mathfrak{B}$. In particular, they proved these two
transformations are different. Taking into consideration their geometric perspective, Theorem \ref{main theorem}
reads as : Brownian symmetrization of a bounded domain does not increase
the Skorokhod energy. In other words,
\[
\Lambda(\mathfrak{B}(U))\leq\Lambda(U).
\]

Theorem \ref{main theorem} persuades us to believe that the Brownian
symmetrization does not increase either the area or the perimeter
\cite{boudabra2024brownian}. If this claim is true then using the area formula \ref{area formula},
it implies in particular that 
\[
\sum_{n=1}^{+\infty}n\vert a_{n}\vert^{2}\leq\sum_{n=1}^{+\infty}n\vert b_{n}\vert^{2}.
\] 

However, the effect of Brownian symmetrization on the principal Dirichlet eigenvalue remains an open question. This problem falls within the broader framework of extremal problems for geometric functionals and shape optimization, a field with a rich history tracing back to Queen Dido and the founding of Carthage (see \cite{bandle2017dido}). In this context, it is worthwhile to study the problem of finding optimal $\mu$-domains with respect to certain properties, such as area or perimeter. More precisely, one can consider the problem

\begin{equation}
\inf_{U:\mu\text{-domain}}\mathcal{J}(U) \label{optimization problem}
\end{equation}
where $\mathcal{J}$ is some functional, and $\mu$ is a distribution with bounded support. Typically, results in this field progress through successive refinements: 
\begin{enumerate}
    \item establishing general bounds on the functional.
    \item proving their sharpness. 
    \item identifying extremal domains.
    \item deriving quantitative stability estimates.
\end{enumerate}
In the context of the planar/conformal Skorokhod embedding problem, extremal questions have been first studied in \cite{mariano2020}, where bounds of types 1) and 2) were established. Additionally, extremal solutions in the classical Skorokhod embedding problem have a long history, with notable results presented in the survey \cite{Obloj2004}. The recent work \cite{beiglbock2017optimal} further clarified these extremal properties by framing the problem within optimal transport. As mentioned in the introduction, one instance of problem \ref{optimization problem} has been treated in \cite{Boudabra2020}, where $\mathcal{J}$ is the principal Dirichlet eigenvalue. For a broader perspective on shape optimization and extremal problems, we refer to the monograph \cite{henrot2017shape}. This direction of research not only deepens the connection between geometry, PDEs, and planar Brownian motion but also opens new avenues for further exploration.

\bibliographystyle{plain}
\bibliography{Skorokhodenergy}

\end{document}